\documentclass[a4paper,11pt,oneside,notitlepage]{amsart}
\usepackage{mabliautoref}
\usepackage{amssymb,amsthm,amsmath}
\RequirePackage[dvipsnames,usenames]{xcolor}
\usepackage{hyperref}
\usepackage{mathtools}
\usepackage[all]{xy}
\usepackage{tikz}
\usepackage{tikz-cd}
\usetikzlibrary{decorations.markings}
\usepackage{esint}
\usepackage{enumitem}
\usepackage{faktor}
\usepackage{chngcntr}
\usepackage{stmaryrd}
\usepackage{minibox}
\usepackage[framemethod=TikZ]{mdframed}

\usepackage[most]{tcolorbox}
\definecolor{darkergray}{rgb}{0.85,0.85,0.85}
\definecolor{lightergray}{rgb}{0.95,0.95,0.95}

\newtcolorbox{activitybox}[1][]{%
    breakable,
    enhanced,
    colback=lightergray,
    boxrule=3pt,
    arc=5pt,
    outer arc=5pt,
    boxsep=10pt,
    colframe=darkergray,
    coltitle=white,
    #1
}

\setlist[itemize,enumerate]{leftmargin=0.9cm}

\mdfsetup{
 skipabove=6pt,
 skipbelow=3pt,
 roundcorner=10pt
}

\usepackage{xparse}

\NewDocumentCommand\SlantLeftDownArrow{O{2.0ex} O{black}}{%
   \mathrel{\tikz[baseline] \draw [<-, line width=0.5pt, #2] (0,0) -- ++(#1,#1);}
}

\NewDocumentCommand\DownArrow{O{2.0ex} O{black}}{%
   \mathrel{\tikz[baseline] \draw [<-, line width=0.5pt, #2] (0,0) -- ++(0,#1);}
}

\NewDocumentCommand\UpArrow{O{2.0ex} O{black}}{%
   \mathrel{\tikz[baseline] \draw [line width=0.5pt, decoration={markings,mark=at position 1 with {\arrow[scale=2, line width=0.25pt]{to}}},  postaction={decorate}, #2] (0,0) -- ++(0,#1);}
}

\hypersetup{
bookmarks,
bookmarksdepth=3,
bookmarksopen,
bookmarksnumbered,
pdfstartview=FitH,
colorlinks,backref,hyperindex,
linkcolor=Sepia,
anchorcolor=BurntOrange,
citecolor=MidnightBlue,
citecolor=OliveGreen,
filecolor=BlueViolet,
menucolor=Yellow,
urlcolor=OliveGreen
}

\DeclareMathAlphabet{\mathchanc}{OT1}{pzc}%
                                 {m}{it}

\newcommand{\bP}{\mathbb{P}}
\newcommand{\bQ}{\mathbb{Q}}

\newcommand{\bZ}{\mathbb{Z}}

\newcommand{\scr}{\mathcal}

\newcommand{\sE}{\scr{E}}
\newcommand{\sF}{\scr{F}}
\newcommand{\sG}{\scr{G}}

\newcommand{\sO}{\scr{O}}

\DeclareMathOperator{\lct}{{lct}}

\DeclareMathOperator{\pre}{{pre}}

\DeclareMathOperator{\rk}{{rk}}

\newcommand{\factor}[2]{\left. \raise 2pt\hbox{\ensuremath{#1}} \right/
        \hskip -2pt\raise -2pt\hbox{\ensuremath{#2}}}

\counterwithin*{equation}{section}
\counterwithin*{equation}{subsection}

\makeatletter
\renewcommand\subsection{
  \renewcommand{\sfdefault}{pag}
  \@startsection{subsection}%
  {2}{0pt}{.8\baselineskip}{.4\baselineskip}{\raggedright
    \sffamily\itshape\small\bfseries
  }}
\renewcommand\section{
  \renewcommand{\sfdefault}{phv}
  \@startsection{section} %
  {1}{0pt}{\baselineskip}{.8\baselineskip}{\centering
    \sffamily
    \scshape
    \bfseries
}}
\makeatother

\renewcommand{\theenumi}{\arabic{enumi}}

\usepackage[left=1.02in,top=1.0in,right=1.02in,bottom=1.0in]{geometry}
\usepackage{mabliautoref}
\usepackage{multirow}
\usepackage{soul}

\setlist[enumerate]{leftmargin=0.8cm}
\setlist[description]{leftmargin=0.0cm}

\title{A note on families of K-semistable log-Fano pairs}
\author{Giulio Codogni}
\address{Dipartimento di Matematica, Università degli Studi Tor Vergata \newline \hspace*{1em} Via della ricerca scientifica 1,	00133 Roma (Italy)}
\email{codogni@mat.uniroma2.it}
\author{Zsolt Patakfalvi}
\address{\'Ecole Polytechnique F\'ed\'erale de Lausanne (EPFL), MA C3 635, Station 8, 1015 Lausanne, Switzerland}
\email{zsolt.patakfalvi@epfl.ch}

\date{\today}

\begin{document}

\maketitle

\begin{abstract}
In this short note, we give an alternative proof of the semipositivity of the Chow-Mumford line bundle for families of K-semistable log-Fano pairs, and of the nefness threeshold for the log-anti-canonical line bundle on families of K-stable log Fano pairs. We also prove a bound on the multiplicity of fibers for families of K-semistable log Fano varieties, which to the best of our knowledge is new.
\end{abstract}

\section{Introduction}

K-polystability is an algebraic stability notion for log-Fano pairs, which, over the complex numbers, is equivalent to the existence of a K\"{a}hler-Einstein metric. Over an algebraically closed field of characteristic zero, K-polystable log-Fano pairs have a good projective moduli space. The Chow-Mumford (CM) line bundle is an ample line bundle on this moduli space. We refer to the introductions of \cite{Codogni_Patakfalvi_Positivity_of_the_CM_line_bundle_for_families_of_K-stable_klt_Fanos} and \cite{Xu_Zhuang_On_positivity_of_the_CM_line_bundle_on_K-moduli_spaces}, to the survey \cite{Xu_Survey_Kstab} and to the recent groundbreaking paper \cite{LXZ} for an exhaustive discussion of these notions and a comprehensive bibliography.

We now recall the definition of the CM line bundle for families of log-Fano pairs over a curve, and in doing so we also establish some notations which will be used through all this article. 

\begin{notation}
\label{notation:basic}
Let $T$ be a smooth projective curve, and let $(X,\Delta)$ be an irreducible, normal pair of dimension $n+1$, both defined over an algebraically closed field $k$ of characteristic zero. Let $f\colon X\to T$ be a flat morphism such that $f_*\mathcal{O}_X=\mathcal{O}_T$. Assume that the relative log-canonical divisor $-K_{X/T}-\Delta$ is $\mathbb{Q}$-Cartier and $f$-ample; in other words, $f$ is a family of log-Fano pairs. The Chow-Mumford line bundle is defined as 
$$\lambda_{CM}=-f_*(-K_{X/T}-\Delta)^{n+1}$$
\end{notation}

We refer to \cite[Section 3]{Codogni_Patakfalvi_Positivity_of_the_CM_line_bundle_for_families_of_K-stable_klt_Fanos}  for the basic properties of $\lambda_{CM}$ and its connection with the other definitions in the literature. Our first result is
\begin{theorem}\label{T:semipos}
In the situation of \autoref{notation:basic}, if there exists a $t$ in $T$ such that $(X_t,\Delta_t)$ is  K-semistable, then $\lambda_{CM}$ is nef.
\end{theorem}
This result was already  proved in \cite[Theorem 1.8]{Codogni_Patakfalvi_Positivity_of_the_CM_line_bundle_for_families_of_K-stable_klt_Fanos} and \cite[Corollary 4.7]{Xu_Zhuang_On_positivity_of_the_CM_line_bundle_on_K-moduli_spaces}. That is, we give the third proof of \autoref{T:semipos}. Our main contribution is that the present proof is particularly short and it uses very little of the theory of filtrations. In fact, all the above proofs  use the Harder-Narasimhan filtration of $f_*\mathcal{O}_X(-qK_{X/T})$. However, our proof uses it in a quite minimalistic manner.

As shown in \cite[Section 10]{Codogni_Patakfalvi_Positivity_of_the_CM_line_bundle_for_families_of_K-stable_klt_Fanos}, \emph{ \autoref{T:semipos} implies that the Chow-Mumford line bundle gives a nef line bundle on the moduli space of K-polystable log-Fano pairs. }

As explained in \cite[Appendix]{Patakfalvi_Zdanowicz_On_the_Beauville_Bogomolov_decomposition_in_characteristic_p_at_least_zero}, the anti-log-canonical bundle is not nef on $X$ unless the family is locally isotrivial. With our methods we can also bound its nefness threeshold as follows:
\begin{theorem}\label{T:nef}
 In the situation of \autoref{notation:basic},  if there exists a $t$ in $T$ such that $(X_t,\Delta_t)$ is K-stable, then 
$$-K_{X/T}-\Delta+\frac{\delta f^*\lambda_{CM}}{(\delta-1)\dim(X)v}F$$
 is nef,  where $\delta$ is the stability threshold of $(X_t,\Delta_t)$, see \autoref{sec:delta}, $F$ is the class of a fiber, and $v= \left(-K_{X_t} - \Delta_t\right)^n$.
\end{theorem}
The above result was proven in \cite[Theorem 1.20]{Codogni_Patakfalvi_Positivity_of_the_CM_line_bundle_for_families_of_K-stable_klt_Fanos}, and in \cite[Corollary 4.10]{Xu_Zhuang_On_positivity_of_the_CM_line_bundle_on_K-moduli_spaces}, but here we give a different proof. The novelty of this new proof is similar as for the proof of \autoref{T:semipos}, which was explained after the statement of \autoref{T:semipos}. We also note that \autoref{T:nef} is used to prove the positivity of the Chow-Mumford line bundle under convenient assumptions, see \cite{Codogni_Patakfalvi_Positivity_of_the_CM_line_bundle_for_families_of_K-stable_klt_Fanos, Posva_Positivity_of_the_CM_line_bundle_for_K-stable_log_Fanos, Xu_Zhuang_On_positivity_of_the_CM_line_bundle_on_K-moduli_spaces}. \autoref{T:semipos} and \autoref{T:nef} have also been recently used in \cite{CTV} to prove slope inequalities for families of K-stable Fano varieties.

Additionally, we prove the following result, which gives a bound on the number of non-reduced fibers of a family of K-semistable Fano varieties, under a semipositivity assumption on the top self-intersection of the anti-canonical bundle of the total space.

\begin{proposition}\label{P:fibers}
In the situation of \autoref{notation:basic}, assume that  $(-K_X-\Delta)^{n+1} \geq 0$ and that $T \cong \bP^1$ (both of these are satisfied if $(X,\Delta)$ is log-Fano). Additionally assume that there exists a $t$ in $T$ such that $(X_t,\Delta_t)$ is K-semistable. Then, denoting by $d_i$ the multiplicity of the non-reduced irreducible fibers, we have
$$  \sum_i \left(1-\frac{1}{d_i} \right) \leq 2 	\,.$$

In particular, there are at most $4$ non-reduced irreducible fibers.
\end{proposition}

\autoref{Ex:sharp} shows that the above result is sharp. \autoref{P:fibers} should be compared with \cite[Corollary 1.17]{Codogni_Patakfalvi_Positivity_of_the_CM_line_bundle_for_families_of_K-stable_klt_Fanos}, where, under similar assumptions, it is given an upper bound for the volume of $(X,\Delta)$.

\subsection{Acknowledgments}

Giulio Codogni is funded by the MIUR {\it ``Excellence Department Project"}, awarded to the Department of Mathematics, University of Rome, Tor Vergata, CUP E83C18000100006, and the  PRIN 2017 {\it ``Advances in Moduli Theory and Birational Classification"}. Giulio Codogni also thanks the organizers of the conference ``Birational Geometry, K\"{a}hler-Einstein metrics and Degenerations", HSE, Moscow,  for their support and hospitality.

Zsolt Patakalvi was partially supported by the following grants: grant \#200021/169639 from the Swiss National Science Foundation,  ERC Starting grant \#804334.

We thank the anonymous referees for their useful suggestions.

\section{Basis type divisors, delta invariants and K-stability}\label{sec:delta}

\emph{For the whole article we work over a fixed algebraically closed field $k$ of characteristic zero.}
Following \cite{Fujita_Odaka_On_the_K-stability_of_Fano_varieties_and_anticanonical_divisors}, we introduce the notion of basis type divisor and \emph{stability threshold} (formerly known as the delta invariant). Let $(Z,\Gamma)$ be a normal projective pair,  that is, $Z$ is a normal projective variety and $\Gamma$ is an effective $\bQ$-divisor on $X$. Let $H$  be a $\mathbb{Q}$-divisor on $Z$. Let $q$ be a positive integer such that $qH$ is Cartier. A $q$-basis type divisor for $(Z,\Gamma;H)$ is a $\bQ$-Cartier divisor $D$ on $Z$ such that there exists a basis $s_1,\dots, s_{N_q}$ of $H^0(Z,qH)$ with 
$$ D=\frac{1}{qN_q}\sum_{i=1}^{N_q}\{s_i=0\}\,. $$
We define the $q$-th stability threshold of the pair $(Z,\Gamma;H)$ as
$$ \delta_q(Z,\Gamma;H):=\inf\big\{\ \lct(Z,\Gamma; D) \  \big| \  \textrm{$D$ is a $q$-basis type divisor} \  \big\} \, , $$
We then define the stability threshold, also known as the delta invariant, as
$$ \delta(Z,\Gamma;H)=\lim_{q\to \infty}\delta_q(Z,\Gamma;H)\,, $$
where the above limit does exist by \cite{Blum_Jonsson_Thresholds__valuations__and_K-stability}.
If $(Z,\Gamma)$ is a log-Fano pair, we let 
$$ \delta(Z,\Gamma):=\delta(Z,\Gamma; -K_Z-\Gamma) \,.$$
We can now give the key definitions
\begin{definition}
A log-Fano pair $(Z,\Gamma)$  is K-semistable if $\delta(Z,\Gamma)\geq1$, it is K-stable if $\delta(Z,\Gamma)> 1$.
\end{definition}


Both K-semistability and K-stability are open conditions in families. For families of log-Fano pairs, the stability threshold of the fiber is a lower-semicontinuous function on the base. If the base field is uncountable, it attains its maximum on the very general geometric fiber. In particular, if the base field is uncountable, one can minimize the coefficient of $F$ in Theorem \ref{T:nef} by taking $t$ a very general point in $T$. If the base field is countable, one can make a field extension and then take a very general point defined over this bigger field.
\section{Harder-Narasimhan filtration and lift of basis type divisors}\label{sec:lifting}

 In the situation of \autoref{notation:basic}, for the values of $q$ such that $-q(K_{X/T}+ \Delta)$ is Cartier, we can look at the following vector bundles on $T$
\begin{equation}
\label{eq:E_q_def}
\sE_q:=f_* \sO_X\big(-q(K_{X/T}+ \Delta ) \big) \,.
\end{equation}
Let 
\begin{equation*}
0=\sF_0^q \subsetneq \sF_1^q \subsetneq \dots \subsetneq \sF_{c_q}^q = \sE_q
\end{equation*}
 be the Harder-Narashiman filtration of $\sE_q$; denote its graded objects $\sF_i^q/\sF_{i-1}^q$ by $\sG_i^q$ and their slopes by $\mu_i^q$. Recall that $\mu_i^q>\mu_{i+1}^q$.

\begin{lemma} \label{basis_type}
 In the situation of \autoref{notation:basic}, fix a closed point $t$ in $T$ such that the fiber $X_t$ over $t$ is a normal variety. For every $q$ divisible enough and for every rational number $\varepsilon>0$ (the divisibility condition on $q$ does not depend on $\varepsilon$), there exists an effective $\bQ$-Cartier divisor  $D^{(q)}_{\varepsilon}$ on $X$ such that
\begin{itemize}
\item $D^{(q)}_{\varepsilon}$ is $\bQ$-linearly equivalent to
$$
M_{q, \varepsilon}= - K_{X/T} - \Delta -\left(\frac{\deg \sE_q}{qN_q }+ \varepsilon \right) X_t 
$$

where $N_q=h^0\big(X_t,-q(K_{X_t} + \Delta_t)\big)$.
\item the restriction of $D^{(q)}_{\varepsilon}$ to the fiber $X_t$ is a $q$-basis type divisor.
\end{itemize}
\end{lemma}
\begin{proof}
 Choose $q$ divisible enough so that $f_* \sO_X\big(-q(K_{X/T}+ \Delta ) \big)$ satisfies cohomology and base-change. This is possible by the relative ampleness assumption on $-K_{X/T} - \Delta$.

Fix an index $i$, and let $a_i$ be a strictly positive integer such that $a_i\mu_i^q$ is an integer. Let $g$ be the genus of $T$. All the slopes of the Harder-Narasimhan filtration of the vector bunde  $\left(\sF_i^q\right)^{\otimes a_i}\otimes \sO_T\big(-(a_i\mu_i^q -2g) t\big)$ are greater or equal to $2g$ (see the proof of \cite[Proposition 5.9]{Codogni_Patakfalvi_Positivity_of_the_CM_line_bundle_for_families_of_K-stable_klt_Fanos}), hence by \cite[Proposition 5.7]{Codogni_Patakfalvi_Positivity_of_the_CM_line_bundle_for_families_of_K-stable_klt_Fanos} the above vector bundle is globally generated.

Take an element $s$ in the fiber 
\begin{equation}
\label{eq:E_q_one_fibre}
\sF_i^q \otimes k(t) \hookrightarrow \sE_q \otimes k(t) \cong H^0\big(T, -q (K_{X_t,}+ \Delta_t) \big),
\end{equation}
which by \autoref{eq:E_q_one_fibre} corresponds to a divisor $\{s=0\} \sim  -q (K_{X_t,}+ \Delta_t)$.
By the above  global generation statement, there exists a global section $\tilde{s}$ of $\sE_q^{\otimes a_i}\otimes \sO_T\big(-(a_i\mu_i^q -2g) t\big)$ which over $t$ equals $s^{\otimes a_i}$ (remark that $\sO_T(t) \otimes k(t) \cong k(t)$ cannonically, so this makes sense). Let $[\tilde{s}]$ be the image of $\tilde{s}$ in $\sE_{qa_i}\otimes \sO_T\big(-(a_i\mu_i^q -2g) t\big)$, via the homomorphism induced by the multiplication map $\sE_q^{\otimes a_i} \to \sE_{qa_i}$.

Using \autoref{eq:E_q_def} and the projection formula we obtain the isomorphism
\begin{equation*}
H^0\Big(T,\sE_{qa_i}\otimes \sO_T\big(-(a_i\mu_i^q -2g) t\big)\Big)
\cong
H^0\big(X,-q a_i(K_{X/T}+ \Delta ) -(a_i\mu_i^q -2g) X_t\big).
\end{equation*}
Hence, we can consider the $\bQ$-Cartier divisor $D_s=\frac{1}{a_i}\{[\tilde{s}]=0\}$ on $X$; its restriction to $X_t$ equals the Cartier divisor $\{s=0\}$, and   by \autoref{eq:E_q_one_fibre} on $X$ we have  
\begin{equation}
\label{eq:D_Q_linear}
 D_s \sim_{\bQ} -q(K_{X/T} + \Delta)-\left(\mu_i^q-\frac{2g}{a_i}\right)X_t.
 \end{equation}
For each integer $1 \leq i \leq c_q$, fix elements $s_{i,j}^q$ for $j=1\dots , \rk(\sG_i^q)$ in $\sF_i^q \otimes t$ whose image in $\sG_i^q \otimes t$ give a basis. For each of them we perform the above construction, obtaining $\bQ$-divisors $D_{i,j}^{(q)}$ on $X$. Let
$$
D^{(q)}_{\pre}:=\frac{1}{qN_q}\sum_{i,j}D_{i,j}^{(q)}.
$$
 By construction, the above sum run over a set of $N_q$ indexes $(i,j)$; in other words, the number of divisors $D_{i,j}^{(q)}$ is equal to the rank of $f_*\O_X(-q(K_{X/T}+\Delta))$.
 
 By \autoref{eq:D_Q_linear} and by the fact that there are $N_q$ appearances of the pairs of indices $(i,j)$, we have
\begin{align*}
D^{(q)}_{\pre} & \sim_{\bQ} \frac{1}{qN_q}\sum_{i,j}\left(-q(K_{X/T} + \Delta) -\left(\mu_i^q-\frac{2g}{a_i}\right)X_t\right)
\\ & =-K_{X/T} - \Delta - \sum_i \left(\frac{\mu_i^q\rk(\sG_i^q)}{qN_q}-\frac{2g \rk(\sG_i^q)}{qN_q a_i}\right)X_t
\\  & =-K_{X/T} - \Delta - \left(\frac{\deg \sE_q}{qN_q} - \sum_i \frac{2g \rk(\sG_i^q)}{qN_q a_i}\right)X_t.
\end{align*}

Apart from the $a_i$'s, everything in $\sum_i \frac{2g \rk(\sG_i^q)}{qN_q a_i}$ is fixed, including the set of indices over which we do the sum. Hence, by choosing $a_i$ big enough we may assume that 
$\sum_i \frac{2g \rk(\sG_i^q)}{qN_q a_i} \leq \varepsilon$. Then we may choose
\begin{equation*}
D^{(q)}_{\varepsilon}=D^{(q)}_{\pre}+ \left(\varepsilon - \sum_i \frac{2g \rk(\sG_i^q)}{qN_q a_i} \right) X_{t'}
\end{equation*}
where $t \neq t' \in T$ is another closed point. To show that this is a good choice, we have to show that $D^{(q)}_{\varepsilon}\big|_{X_t}$ is a basis type divisor. Indeed, the restriction of each $D_{i,j}^{(q)}$ to $X_t$ gives an element of a basis of the linear system $|-q(K_{X_t} + \Delta_t)|$, hence the restriction of $D^{(q)}_{\varepsilon}$ gives a $q$-basis type divisor.

\end{proof}

\section{Nefness threshold}
The following lemma is a consequence of \cite[Thm 1.13]{Fujino_Semi_positivity_theorems_for_moduli_problems}
\begin{lemma}\label{fuj}
In the situation of \autoref{notation:basic},
let $\Gamma$ be an effective $\bQ$-Cartier $\bQ$-divisor on $X$ such that the pair $(X_t, \Delta_t + \Gamma_t)$ is klt for some  closed point  $t\in T$ and $K_{X/T}+ \Delta + \Gamma$ is $f$-ample, then $K_{X/T}+  \Delta +\Gamma$ is nef.
\end{lemma}
\begin{proof}
By  \cite[Thm. 1.13]{Fujino_Semi_positivity_theorems_for_moduli_problems}, the vector bundle $f_*\sO_X(q(K_{X/T}+  \Delta +\Gamma))$ is nef for all $q$ big and divisible enough. The evaluation map $f^*f_*\sO_X(q(K_{X/T}+  \Delta +\Gamma))\to \sO_X(q(K_{X/T}+  \Delta +\Gamma))$ is surjective for all $q$ divisible enough, as $K_{X/T}+\Delta+\Gamma$ is $f$-ample. We conclude that $q(K_{X/T}+\Delta +\Gamma)$ and hence $ K_{X/T}+ \Delta +\Gamma$ is nef.
\end{proof}

\begin{proof}[Proof of  \autoref{T:nef}]

We keep the notation of \autoref{basis_type}. Applying this lemma and using the definition of K-stability yields that for all rational number $\varepsilon' \in (0,\delta-1)$, there exists an integer $q(\varepsilon')$ such that for all positive integers $q(\varepsilon') \big| q$ the pair $\Big(X_t,(1+\varepsilon') D^{(q)}_t+\Delta_t\Big)$ is klt. 
 Fix such integer $q(\varepsilon')$.
By \autoref{fuj}, the $\bQ$-Cartier divisor \begin{equation*}
N_{\varepsilon'}=K_{X/T}+\Delta+(1+\varepsilon')D^{(q(\varepsilon'))}_{\varepsilon}  \sim_{\bQ} - \varepsilon'( K_{X/T}+ \Delta ) - (1+ \varepsilon')  \left(\frac{\deg \sE_{q(\varepsilon')}}{q(\varepsilon')N_{q(\varepsilon')} }+ \varepsilon \right) X_t 
\end{equation*}
 is nef on $X$.  Hence, also $\frac{N_{\varepsilon'}}{\varepsilon'}$ is nef. Since, this is true for every $\bQ \ni \varepsilon >0$, by limiting with $\varepsilon$ to $0$  we obtain that
  \begin{equation}
  \label{eq:nef_thhold_pre}
-( K_{X/T}+ \Delta ) -  \frac{(1+ \varepsilon')\deg \sE_{q(\varepsilon')}}{\varepsilon' q(\varepsilon')N_{q(\varepsilon')} } X_t 
 \end{equation}
is nef. Next we note that if we limit with $\varepsilon'$ to $\delta-1$,  then we may choose that at the same time $q(\varepsilon')$ limits to $\infty$. Indeed, this is possible, since $q(\varepsilon')$ can be replaced by each of its multiples.  Additionally, $\lim_{q\to \infty}\frac{\deg \sE_q}{qN_q }=-\frac{\deg(\lambda_{CM})}{\dim(X)v}$, see \cite[Section 3]{Codogni_Patakfalvi_Positivity_of_the_CM_line_bundle_for_families_of_K-stable_klt_Fanos}. So, by limiting  $\varepsilon'$ in \autoref{eq:nef_thhold_pre} to $\delta-1$ we obtain that
 \begin{equation*}
-( K_{X/T}+ \Delta ) -  \frac{\delta \deg(\lambda_{CM})}{(\delta-1)\dim(X)v } X_t 
 \end{equation*}
 is nef.
 
\end{proof}

\section{Semipositivity}

\begin{proof}[Proof of \autoref{T:semipos}]
Assume by contradiction that $\deg(\lambda_{CM})<0$. We keep the notation of \autoref{basis_type}. 

Let $a$ be a positive integer such that $E=-K_{X/T}-\Delta+aX_t$ is ample on $X$.  Let 
 $c, \varepsilon>0$ be rational numbers such that for all $q$ divisible enough we have
$$(1-c) \left(\frac{\deg(\lambda_{CM})}{\dim(X)v}+ \varepsilon \right)+ca<0 \,.$$

The K-semistability assumption implies that $\delta(X_t,\Delta_t)\geq 1$, so for all $q$ divisible enough we have $\delta_q(X_t,\Delta|_{X_t})>1-c$. By the definition of $\delta_q$ in terms of log canonical threshold of $q$-basis type divisors, we have that $\left(X_t,(1-c)\big(D^{(q)}_{\varepsilon}\big)_t+\Delta|_{X_t}\right)$ is klt for all $q$ divisible enough. This yields a contradiction with \cite[Prop. 7.2]{Codogni_Patakfalvi_Positivity_of_the_CM_line_bundle_for_families_of_K-stable_klt_Fanos} as we can write
$$
(1-c)D^{(q)}_{\varepsilon}+cE \sim_{\bQ} -K_{X/T}-\Delta+\left( (1-c) \left(\frac{\deg(\lambda_{CM})}{\dim(X)v}+ \varepsilon \right)+ca\right) X_t \,.
$$
\end{proof}

\section{Bound on the multiplicity of the fibers}
\begin{proof}[Proof of \autoref{P:fibers}]
Let $d_iF_i$ be the non-reduced fibers of $f$, and $d$ a common multiple of the $d_i$. Let $\tau \colon S \to T$ be the degree $d$ cover of $T$ totally ramified at the points corresponding to the non-reduced fibers. Denote by $Y$ the normalization of base change $X_S$, and by $\sigma \colon Y \to X$ and $g\colon Y\to S$ the induced maps. Let $\Delta_Y:= \sigma^* \Delta$ (for the pull-back of a Weil divisor via a finite map between normal varieties see \cite[proof of Proposition 5.20]{Kollar_Mori_Birational_geometry_of_algebraic_varieties}). We have

\begin{equation*}
-K_{Y/S} -\Delta_Y - \sum_i \frac{d}{d_i}(d_i -1) F_i = -\sigma^* (K_{X/T}+\Delta) \,.
\end{equation*}
As $d_i F_i \sim_f 0$, we have $F_i \sim_{f,\bQ} 0$, and hence, $-K_{Y/S} - \Delta_Y$ is $g$-ample and $g$ is a family of log-Fano varieties. As the generic fiber of $g$ is isomorphic to the generic fiber of $f$, it is K-semistable, hence the Chow-Mumford line bundle of $g$ is nef, in other words $(-K_{Y/S}-\Delta_Y)^{n+1}\leq 0$.

We now make the following direct computation.
\begin{align*}
-(n+1)  \left(-K_{Y_t}-(\Delta_Y)_t\right)^n & \cdot \sum_i \left(d-\frac{d}{d_i} \right) 
\\ & \geq
(-K_{Y/S}-\Delta_Y)^{n+1}  - (n+1)  \left(-K_{Y_t}-(\Delta_Y)_t\right)^n \cdot \sum_i \left(d-\frac{d}{d_i} \right)
\\ & = \left( -K_{Y/S}-\Delta_Y - \sum_i \frac{d}{d_i}(d_i -1) F_i \right)^{n+1} 
\\ & = \left(-\sigma^*\left( K_{X/T}+\Delta \right)\right)^{n+1} 
\\ & = d (-K_{X/T}-\Delta)^{n+1} 
\\ & = d (-K_X -\Delta + f^* K_T)^{n+1}
\\ &  = d (-K_X-\Delta)^{n+1} - 2 d(n+1) \left(-K_{X_t}-\Delta_t\right)^n  
\\ & \geq -2 d(n+1) \left(-K_{X_t}-\Delta_t \right)^n \,.
\end{align*}
As $(X_t,\Delta_t)$ and $(Y_t,(\Delta_Y)_t)$ are isomorphic for generic $t$, we conclude that
\begin{equation*}
 \sum_i \left(1-\frac{1}{d_i} \right) \leq 2 \,.
\end{equation*}
\end{proof}

\begin{example}\label{Ex:sharp}
This example shows that \autoref{P:fibers} is sharp, and the condition $(-K_X)^{n+1}\geq 0$ is necessary. Let $C$ be a genus $g$ hyperelliptic curve, and $\iota$ the hyperelliplict involution. Let $X$ be the quotient of $\bP^1\times C$ by  $G= \bZ/ 2 \bZ$, which acts on $C$ by $\iota$ and on $\bP^1$ as the standard involution. Consider the morphism $f\colon X\to C/\iota \cong \bP^1=:T$. This map is a $\bP^1$-bundle, so that the generic fiber is K-semistable, and it has $2g+2$ non-reduced fibers of multiplicity $2$. The condition $(-K_X)^2\geq 0$ is fulfilled if and only if $g\leq 1$.

\end{example}

\bibliographystyle{skalpha}
\bibliography{includeNice}

\end{document}